\documentclass[11pt,reqno]{amsart}

\usepackage{amsmath,amssymb,mathtools}
\usepackage[margin=1.05in]{geometry}
\usepackage[hidelinks]{hyperref}

\newtheorem{theorem}{Theorem}[section]
\newtheorem{proposition}[theorem]{Proposition}
\newtheorem{lemma}[theorem]{Lemma}
\newtheorem{corollary}[theorem]{Corollary}

\newtheorem{conj}[theorem]{Conjecture}

\theoremstyle{definition}
\newtheorem{definition}[theorem]{Definition}
\newtheorem{remark}[theorem]{Remark}

\newcommand{\F}{\mathbb F}

\newcommand{\rank}{\operatorname{rank}}

\newcommand{\BR}{\operatorname{BR}}

\title{Linear matroid products: a synthetic approach}
\author{Mykhaylo Tyomkyn}
\address{Department of Applied Mathematics, Faculty of Mathematics and Physics, Charles University}
\curraddr{}
\email{tyomkyn@kam.mff.cuni.cz}
\thanks{Supported by GA\v{C}R grant 25-17377S}

\begin{document}

\begin{abstract}
The recently established duality between tensor, symmetric and exterior products of uniform matroids on the one hand, and abstract (bi-)rigidity matroids on the other hand, connects theorems and open questions from both areas, previously thought unrelated.  
In particular, a 1981 question of Mason asks if any two uniform matroids admit a freest product. Via duality, its reformulation due to Cruickshank, Jackson, Jord\'an and Tanigawa asks if the generic birigidity matroid is the freest abstract birigidity matroid for its parameters. A related conjecture of Jackson and Tanigawa asks if the generic $2$-hyperconnectivity matroid is the freest $\{K_4,K_{3,3}\}$-matroid. We show that these problems can be addressed effectively, and often solved completely for the class of linearly representable matroids.  

To this end, we introduce star-basis normal forms that allow to compare representations of abstract (bi-)rigidity matroids over the same field, resulting in a refinement of the weak order relation on the underlying matroids. Utilizing it, we prove that 
\begin{itemize}
\item The generic birigidity matroid is the freest linearly representable abstract $(a,b)$-birigidity matroid. This confirms Mason's conjecture for linear matroids. 
\item Every representable abstract $2$-rigidity matroid admits a rigidity matrix representation. This is a strengthening of the maximality property of the generic $2$-rigidity matroid.
\item  Every representable abstract $2$-rigidity matroid in which every copy of $K_{3,3}$ is a circuit admits a hyperconnectivity matrix representation. It follows that the generic rigidity and hyperconnectivity families $\mathcal{R}_2$ and $\mathcal{H}_2$ are the only linearly representable $2$-rigidity families.
\end{itemize}
In a follow-up paper we will address $d$-dimensional rigidity in general, and the $d=3$ case in particular. 
\end{abstract}

\maketitle

\section{Introduction}

Matroid products were introduced by Lov\'asz~\cite{Lovasz77} and Mason~\cite{Mas81}, with the aim to capture combinatorially multilinear product constructions on vector spaces. Their areas of application encompass multilinear algebra~\cite{Mas81}, matrix completion~\cite{KTT15}, coding theory~\cite{gopalan2017maximally}, rigidity~\cite{Bra24}, and linear rank inequalities~\cite{BGILPS25}.

Let \(M_1\) and \(M_2\) be loopless matroids on ground sets \(E_1\) and \(E_2\), respectively. A matroid $M$ on \(E_1\times E_2\), whose elements may be viewed as the edges of the complete bipartite graph between \(E_1\)
and \(E_2\), is a \emph{product} of $M_1$ and $M_2$ if it satisfies the following axioms.

\begin{enumerate}
	\item[(P1)] For every $e_1\in E_1$, the restriction of $M$ to $\{e_1\}\times E_2$ is naturally isomorphic to $M_2$. For every $e_2\in E_2$, the restriction of $M$ to $E_1\times \{e_2\}$ is naturally isomorphic to $M_1$. 
	\item[(P2)] For all flats $F_1$ of $M_1$ and $F_2$ of $M_2$, $F_1\times F_2$ is a flat of $M$.
\end{enumerate}
If in place of (P2) one imposes the stronger condition $r_M(X_1\times X_2)=r_{M_1}(X_1)r_{M_2}(X_2)$ for all $X_1\subseteq E_1,X_2\subseteq E_2$, then $M$ is a \emph{tensor product}. While Lov\'asz~\cite{Lovasz77} showed that any two matroids admit a product, 
Las Vergnas~\cite{LV81} proved that tensor products need not always exist (see also~\cite{BGILMS25, BGILPS25} for recent developments). 

When
\(M_1\) and \(M_2\) are represented over a common field by vectors
\((u_e)_{e\in E_1}\) and \((v_{e'})_{e'\in E_2}\), the vectors
\[
u_e\otimes v_{e'},
\qquad (e,e')\in E_1\times E_2,
\]
represent a tensor product of $M_1$ and $M_2$. In the important special case when $M_1$ and $M_2$ are uniform matroids represented by generic vectors over a function field, we speak of the \emph{generic tensor product} of two uniform matroids in a given characteristic. When the characteristic of a generic tensor product is not mentioned, it is assumed to be $0$.
Since the characteristic $0$ function field embeds into $\mathbb{C}$, the generic tensor product is $\mathbb{C}$-representable.

We compare matroids on a common ground set by the weak order. For matroids $M$ and $N$,  we say that $N$ is
\emph{freer} than $M$, and  write
$M\preceq N$, if each independent set of $M$ is independent in $N$ (equivalently,
each set dependent in $N$ is dependent in $M$). Mason~\cite{Mas81} asked whether for any two matroids $M_1$ and $M_2$, the set of their products has a unique maximum in the weak order (their \emph{freest} product).  When $M_1$ and $M_2$ admit a tensor product, the conjectured maximum must be a tensor product. Las Vergnas~\cite{LV81} answered Mason's question in the negative. However, the natural case of $M_1$ and $M_2$ being uniform matroids remains open, and its restriction to tensor products has established itself in modern literature under the name \emph{Mason's conjecture} (see~\cite{CJJT25, Bra25}).
\begin{conj}[Mason's conjecture]
The generic tensor product of the uniform matroids ${U}_{m-a,m}$
and ${U}_{n-b,n}$ is their freest matroid tensor product. 
\end{conj}	
A bridge to rigidity theory was recently discovered by Brakensiek, Dhar,
Gao, Gopi, and Larson~\cite{Bra24}. They identified the dual
of the generic tensor product of two uniform matroids as the generic
bipartite rigidity matroid in the same characteristic. In a similar fashion they related the generic symmetric and exterior-power matroids to rigidity and hyperconnectivity matroids, respectively. Subsequently, it was observed by Cruickshank, Jackson, Jord\'an, and Tanigawa~\cite{CJJT25} that the duals of (matroidal) tensor products of uniform matroids are precisely the abstract birigidity matroids. Hence, problems about products of uniform matroids can be translated into those concerning abstract (bi-)rigidity matroids, and vice versa. 

By this duality, the authors of~\cite{CJJT25} reformulated Mason's conjecture as the assertion that the generic birigidity matroid is the freest abstract $(a,b)$-birigidity matroid on a given vertex set. This reformulation also has algorithmic significance.
Brakensiek, Chen, Dhar, and Zhang~\cite{Bra25} constructed an explicit potential matroid lying above the generic birigidity matroid in the weak order, and showed that Mason's conjecture would yield a deterministic polynomial-time
independence algorithm in the generic birigidity and tensor product matroids. Our first result confirms Mason's conjecture within the class of linearly representable matroids.

\begin{theorem}
	\label{thm:global-maximality}
	The generic birigidity matroid
	$
\BR_{a,b}(m,n)
$
	is the freest linearly representable abstract $(a,b)$-birigidity matroid on $[m]\times [n]$. Every representable abstract $(a,b)$-birigidity matroid admits a representation by a standard birigidity matrix over the same field.
\end{theorem}
\begin{corollary}\label{cor:tensor-freest}
The generic tensor product of the uniform matroids ${U}_{m-a,m}$
and ${U}_{n-b,n}$ is their freest linearly representable matroid tensor product. 
\end{corollary}
Thus the problem of polynomially recognizing independence in the generic tensor products of uniform matroids reduces to that of linear representability of a concrete family of matroids.

\begin{corollary}\label{cor:potential-recognition}
Suppose that, for every $a,b\geq 2$, $m\geq a+1$, and
$n\geq b+1$, the potential matroid of Brakensiek et
al.~\cite{Bra25} associated with $\BR_{a,b}(m,n)$ is linearly
representable. Then, for every fixed $a,b\geq 2$, the independence
problem for $\BR_{a,b}(m,n)$ and for the generic tensor product
$U_{m-a,m}\otimes U_{n-b,n}$ can be decided deterministically in time
polynomial in $m$ and $n$.
\end{corollary}
To prove Theorem~\ref{thm:global-maximality}, we take an arbitrary representation $\mathcal{A}$ of an abstract $(a,b)$-birigidity matroid and, breaking the symmetry, express $\mathcal{A}$ in a designated `star basis'. We derive certain proportionality relations for its coefficients, and show that under these conditions the star-basis representation can be converted into a standard birigidity matrix form. Since the generic birigidity matroid is given by a birigidity matrix with generic parameters, Theorem~\ref{thm:global-maximality} follows. 

We apply a similar idea in two-dimensional rigidity to obtain the following strengthening of the maximality property of the generic $2$-rigidity matroid, for linear matroids.
\begin{theorem}\label{thm:rigidity-realization}
Every linearly representable abstract $2$-rigidity matroid can be represented by a standard two-dimensional infinitesimal rigidity matrix over the same field.
\end{theorem}

Imposing additionally that every \(K_{3,3}\) is a circuit results in more restrictions on the parameters in the star-basis form. This yields the analogous result for Kalai's $2$-hyperconnectivity matroid.

\begin{theorem}\label{thm:main-hyper-realization}
Every linearly representable abstract $2$-rigidity matroid \(M\) on $n\geq 6$ vertices, in which every copy of $K_{3,3}$ is a circuit, can be represented by a standard two-dimensional hyperconnectivity matrix over the same field.
\end{theorem}
A conjecture of Jackson and Tanigawa~\cite{Jac24} states that the generic $2$-hyperconnectivity matroid is the freest among graph matroids in which every $K_4$ and $K_{3,3}$ are circuits. Theorem~\ref{thm:main-hyper-realization} shows that any potential counterexample must be either non-representable or of lower total rank. A strong counterexample that is freer than the generic $2$-hyperconnectivity matroid has to be non-representable. 

Kalai~\cite{Kal85} defined $d$-rigidity families under the name ``hypergraphic sequences'' and suggested their further study. Using modern language, a $d$-rigidity family is a family $\mathcal{F}$ of finite graphs\footnote{Throughout this paper, graphs are identified with their sets of edges, and hence are assumed to have no isolated vertices.} such that for every $n$, the collection $\mathcal{F}^n$ of all (labelled) copies of graphs from $\mathcal{F}$ with vertices in $[n]$ forms an independence family of an abstract $d$-rigidity matroid (in particular, every $d$-degenerate graph is in $\mathcal{F}$, while $K_{d+2}$ is not). For $d=1$ there is a unique such family, namely the forests. Hence, in general, $d$-rigidity families can be viewed as natural generalizations of the forests, capturing $d$-degenerate graphs matroidally. 

For $d=2$ the two known $d$-rigidity families are the generic rigidity family $\mathcal{R}_2$ and the hyperconnectivity family $\mathcal{H}_2$, and it was conjectured in~\cite{Tyo26} that no further $2$-rigidity families exist. Using the characteristic-independent description of $2$-hyperconnectivity and the clarification of the role of \(K_{3,3}\) in \(2\)-rigidity families~\cite{Tyo26}, we obtain the following conclusion, further supporting this conjecture.

\begin{theorem}\label{thm:representable-families}
	Suppose $\mathcal{F}$ is a $2$-rigidity family different from $\mathcal{R}_2$ and $\mathcal{H}_2$.
Then there exists $n_0=n_0(\mathcal F)$ such that, for every
$n\geq n_0$, the matroid whose independence family is $\mathcal F^n$
is not linearly representable.
\end{theorem}
The paper follows the order of these results.  Section~\ref{sec:birigidity}
develops the birigidity normal form and proves Theorem~\ref{thm:global-maximality}.  Sections~\ref{sec:two-rigidity} and~\ref{sec:hyperconnectivity} apply the star-basis method to two-dimensional rigidity and hyperconnectivity, respectively.

\subsection*{Statement on AI use}
The idea to use star bases in rigidity and birigidity is my own and predates the wide-scale use of AI tools in mathematical research. I carried out the birigidity calculations by hand and reached the conclusion of Theorem~\ref{thm:global-maximality}. I then used ChatGPT 5.5 Pro to verify and simplify the argument. In two-dimensional rigidity, I first attempted the calculations on my own but was not able to carry them out, as I did not realize that only a small subset of $K_4$-conditions sufficed to be checked. This transpired in the course of a dialogue with ChatGPT 5.5 Pro, which resulted in the proofs of Theorem~\ref{thm:rigidity-realization} and~\ref{thm:main-hyper-realization}. The idea of Theorem~\ref{thm:representable-families} and of its proof are mine, ChatGPT 5.6 Sol carried out the steps of the proof. After the first draft of the paper was written by ChatGPT 5.6 Sol, I have carefully re-worked it and re-written almost every word of it in human language. I revised for typos with ChatGPT 6 Astra. I take full responsibility for the content and wording of this paper.

\section{Birigidity}\label{sec:birigidity}
Throughout the paper, if $p=0$, put $k_p=\mathbb Q$, and if $p$ is
prime, put $k_p=\mathbb F_p$. We recall a well-known fact about matroid representations.
\begin{lemma}\label{lem:scalar-extension}
	Suppose $K\subseteq L$ are fields, and let $A$ be a matrix with entries
	in $K$. Then the column matroid of $A$ is the same over $K$ and $L$.
\end{lemma}
\begin{proof}
	A set of $r$ columns is independent over either field if and only if
	some square submatrix formed from these columns and $r$ rows has
	non-zero determinant. Every such determinant belongs to $K$, so whether it is zero does not change when interpreting the coefficients over a larger field. Hence every set of columns is independent over $K$ if and only if it is independent over $L$.
\end{proof}
\subsection{Abstract birigidity matroids and the standard matrix}

Let $a,b,m,n$ be integers satisfying
\[
a,b\geq 2,\qquad m\geq a+1,\qquad n\geq b+1.
\]
Put
\[
L=[m],\qquad R=[n].
\]
The edge set of $K_{m,n}$ is identified with $L\times R$, where we interpret $L$ and $R$ as `left' and `right' vertices, respectively.  An edge is
written $ij$, where $i\in L$ and $j\in R$.  

Set
\begin{equation*}
r_{a,b}(m,n)=mn-(m-a)(n-b)
= an+bm-ab.
\end{equation*}
\begin{definition}\label{def:abstract-birigidity}
An \emph{abstract $(a,b)$-birigidity matroid} on $L\times R$ is a
matroid $M$ satisfying the following conditions:
\begin{enumerate}
\item $\rank M=r_{a,b}(m,n)$;
\item for every $A\subseteq L$ and $D\subseteq R$ with
      $|A|=a+1$ and $|D|=b+1$, the edge set $A\times D$ is an
      $M$-circuit.
\end{enumerate}
\end{definition}
\noindent
We next recall the standard birigidity matrix representation. Let $\F$ be a field.
Given, for every $i\in L$, a vector
\[
u_i=(u_{i,1},\ldots,u_{i,a})\in\F^a,
\]
and, for every $j\in R$, a vector
\[
v_j=(v_{j,1},\ldots,v_{j,b})\in\F^b,
\]
write
\[
u=(u_i)_{i\in L},\qquad v=(v_j)_{j\in R}.
\]
\begin{definition}[The standard $(a,b)$-birigidity matrix, transposed]
\label{def:standard-matrix}

Given $u$ and $v$ as above, the matrix
\[
\mathcal R_{a,b}(u,v)
\]
is the $(bm+an)\times mn$-matrix as follows. Its columns $h_{ij}$ are indexed by $ij\in L\times R$.  Its rows are labelled as
\[
U_{i,q}\quad(i\in L,\ q\in [b])
\]
and
\[
V_{p,j}\quad(p\in[a],\ j\in R).
\]
For $i\in L$ and $j\in R$, the entries of the column $h_{ij}$
are
\begin{equation*}
(\mathcal R_{a,b}(u,v))_{U_{i,q},ij}=v_{j,q}
\quad(1\leq q\leq b)
\end{equation*}
and
\begin{equation}\label{eq:standard-column-right}
(\mathcal R_{a,b}(u,v))_{V_{p,j},ij}=u_{i,p}
\quad(1\leq p\leq a).
\end{equation}
All other entries in $h_{ij}$ are zero.
\end{definition}
For an illustration, the matrix below describes the case $(a,b,m,n)=(2,2,3,3)$.
\[
\begin{array}{c|ccccccccc}
	&h_{11}&h_{12}&h_{13}&h_{21}&h_{22}&h_{23}
	&h_{31}&h_{32}&h_{33}\\ \hline
	U_{1,1}&v_{1,1}&v_{2,1}&v_{3,1}&0&0&0&0&0&0\\
	U_{1,2}&v_{1,2}&v_{2,2}&v_{3,2}&0&0&0&0&0&0\\
	U_{2,1}&0&0&0&v_{1,1}&v_{2,1}&v_{3,1}&0&0&0\\
	U_{2,2}&0&0&0&v_{1,2}&v_{2,2}&v_{3,2}&0&0&0\\
	U_{3,1}&0&0&0&0&0&0&v_{1,1}&v_{2,1}&v_{3,1}\\
	U_{3,2}&0&0&0&0&0&0&v_{1,2}&v_{2,2}&v_{3,2}\\
	V_{1,1}&u_{1,1}&0&0&u_{2,1}&0&0&u_{3,1}&0&0\\
	V_{2,1}&u_{1,2}&0&0&u_{2,2}&0&0&u_{3,2}&0&0\\
	V_{1,2}&0&u_{1,1}&0&0&u_{2,1}&0&0&u_{3,1}&0\\
	V_{2,2}&0&u_{1,2}&0&0&u_{2,2}&0&0&u_{3,2}&0\\
	V_{1,3}&0&0&u_{1,1}&0&0&u_{2,1}&0&0&u_{3,1}\\
	V_{2,3}&0&0&u_{1,2}&0&0&u_{2,2}&0&0&u_{3,2}
\end{array}
\]
The convention in which the entries in \eqref{eq:standard-column-right}
are $-u_{i,p}$ is also used.  Multiplying every row $V_{p,j}$ by $-1$
transforms one convention into the other, so the two conventions
represent the same matroid.
\begin{definition}[The generic $(a,b)$-birigidity matroid]\label{def:generic-matroid}
	Suppose $p=0$ or $p$ is prime. Let
	\[
	\{X_{i,r}:i\in L,\ r\in[a]\}
	\cup
	\{Y_{j,s}:j\in R,\ s\in[b]\}
	\]
	be jointly algebraically independent indeterminates over $k_p$. Put
	\[
	\mathbf X_i=(X_{i,1},\ldots,X_{i,a})\quad(i\in L),
	\qquad
	\mathbf Y_j=(Y_{j,1},\ldots,Y_{j,b})\quad(j\in R),
	\]
	and write
	\[
	\mathbf X=(\mathbf X_i)_{i\in L},
	\qquad
	\mathbf Y=(\mathbf Y_j)_{j\in R}.
	\]
	The \emph{generic $(a,b)$-birigidity matroid in characteristic $p$},
	denoted by $\BR_{a,b}^{(p)}(m,n)$, is the column matroid of
	$\mathcal R_{a,b}(\mathbf X,\mathbf Y)$ over the rational function field
	\[
	k_p\bigl(X_{i,r},Y_{j,s}:i\in L,\ r\in[a],\
	j\in R,\ s\in[b]\bigr).
	\]
	When $p=0$, write
	\[
	\BR_{a,b}(m,n)=\BR_{a,b}^{(0)}(m,n).
	\]
\end{definition}

Suppose $\F$ is a field of characteristic $p$, and take the
indeterminates $X_{i,r}$ and $Y_{j,s}$ in
Definition~\ref{def:generic-matroid} to be jointly algebraically
independent over $\F$. By Lemma~\ref{lem:scalar-extension}, the matrix
$\mathcal R_{a,b}(\mathbf X,\mathbf Y)$ also represents
$\BR_{a,b}^{(p)}(m,n)$ over
\[
\F\bigl(X_{i,r},Y_{j,s}:i\in L,\ r\in[a],\
j\in R,\ s\in[b]\bigr).
\]
\noindent
We recall some well-known facts about birigidity matroids.
\begin{proposition}\label{prop:generic-is-abstract}
	\leavevmode
	\begin{enumerate}
		\item[(i)] $\BR_{a,b}(m,n)$ is an abstract $(a,b)$-birigidity matroid.
		\item[(ii)] Let $M$ be a matroid admitting a standard $(a,b)$-birigidity matrix representation $\mathcal R_{a,b}(u,v)$ over a field $\F$ of characteristic $p$. Then 
		$$M\preceq \BR_{a,b}^{(p)}(m,n).
		$$
		\item[(iii)] For every prime $p$,
		\begin{equation*}
			\BR_{a,b}^{(p)}(m,n)
			\preceq
			\BR_{a,b}(m,n).
		\end{equation*}
	\end{enumerate}
\end{proposition}
Thus, the generic birigidity matroids arise, as their name suggests, by choosing the parameters of the standard birigidity matrix generically. We note that non-generic parameter choices may result in a different matroid, which need not be an abstract $(a,b)$-birigidity matroid.
\subsection{The normalized star-basis representation}

Put
\[
L_0=[a],\qquad R_0=[b],
\]
and
\[
L_1=L\setminus L_0,\qquad R_1=R\setminus R_0.
\]
Define the star edge set
\begin{equation*}
\mathcal S=(L_0\times R)\cup(L\times R_0).
\end{equation*}
Its cardinality is
\begin{align*}
|\mathcal S|
&=an+mb-ab=r_{a,b}(m,n).
\end{align*}
\begin{lemma}\label{lem:star-is-basis}
Let $M$ be an abstract $(a,b)$-birigidity matroid on $L\times R$.
Then $\mathcal S$ is a basis of $M$.
\end{lemma}
\begin{proof}
Let $i\in L_1$ and $j\in R_1$.  The edge set
$
(L_0\cup\{i\})\times(R_0\cup\{j\})
$
is an $M$-circuit.  Every edge in it except
$ij$ belongs to $\mathcal S$.  Therefore
$ij\in\operatorname{cl}_M(\mathcal S)$.  This holds for every
$i\in L_1$ and every $j\in R_1$, so $\mathcal S$ spans $M$.  Since
$|\mathcal S|=\rank M$, it is a basis.
\end{proof}
Let $M$ now be representable over a field $\F$. Since $\mathcal S$ is a basis, $M$ admits an $\F$-representation as the column matroid of an $r_{a,b}(m,n)\times mn$-matrix, whose rows and columns are indexed by $\mathcal S$ and $L\times R$, respectively, and the columns indexed by $\mathcal{S}$ form an identity matrix.  We call this a \emph{star-basis representation} with respect to $\mathcal S$. Note that in general a star-basis representation is not unique. Indeed, any representation of $M$ as a column matroid of a $r_{a,b}(m,n)\times mn$ matrix $\mathcal{A}$ can be converted to star-basis by multiplying $\mathcal A$ on the left with the inverse of the $\mathcal S \times \mathcal S$-indexed square submatrix. 

However, in any star-basis representation, certain relations between coefficients must always be satisfied. By definition, its columns indexed by $\mathcal{S}$ form an identity matrix. As for the remaining coefficients, we show that they can be parametrized by a small set of parameters. To this end, we need one technical lemma.

For each $s\in\mathcal S$, denote by
$\mathbf e_s\in\F^{\mathcal S}$ the vector whose entry in coordinate $s$ is $1$, and $0$ otherwise. For any $(i,j)\in L_1\times R_1$, since $(L_0\cup\{i\})\times(R_0\cup\{j\})$ is the
fundamental circuit of $ij$ with respect to $\mathcal S$, there are
non-zero coefficients
\[
\gamma_{pq}^{ij},\quad \alpha_p^{ij},\quad \beta_q^{ij}\in\F^{\times} \ \ (p\in L_0,\ q\in R_0)
\]
such that the column $C_{ij}$ indexed by $ij$ satisfies
\begin{equation}\label{eq:initial-column}
	C_{ij}
	=
	\sum_{p\in L_0}\sum_{q\in R_0}\gamma_{pq}^{ij}\mathbf e_{pq}
	+
	\sum_{p\in L_0}\alpha_p^{ij}\mathbf e_{pj}
	+
	\sum_{q\in R_0}\beta_q^{ij}\mathbf e_{iq}.
\end{equation}
\begin{lemma}\label{lem:proportionality}
Let $M$ be an $\F$-representable abstract $(a,b)$-birigidity matroid on
$L\times R$, in star-basis representation with respect to $\mathcal S$,
so that~\eqref{eq:initial-column} is satisfied. Then for all $p\in L_0,q\in R_0, j\in R_1$ and distinct $i,k\in L_1$ we have
\begin{equation}\label{eq:left-proportionality}
	\frac{\gamma_{pq}^{ij}}{\alpha_p^{ij}}
	=
	\frac{\gamma_{pq}^{kj}}{\alpha_p^{kj}}.
\end{equation}
And, symmetrically, for all $p\in L_0,q\in R_0, i\in L_1$ and distinct $j,\ell \in R_1$ we have 	
\begin{equation}\label{eq:right-proportionality}
	\frac{\gamma_{pq}^{ij}}{\beta_q^{ij}}
	=
	\frac{\gamma_{pq}^{i\ell}}{\beta_q^{i\ell}}.
\end{equation}
\end{lemma}	
\begin{proof}
Let $p\in L_0$, let
$i,k\in L_1$ be distinct, and let $j\in R_1$.  Put
\begin{equation}\label{eq:left-two-one}
	\mathcal Q
	=
	\bigl((L_0\setminus\{p\})\cup\{i,k\}\bigr)
	\times(R_0\cup\{j\}).
\end{equation}
The only edges of $\mathcal Q$ outside $\mathcal S$ are $ij$ and $kj$.
Since $\mathcal Q$ is a circuit, there are non-zero scalars
\[
\lambda,\mu\in\F
\quad\text{and}\quad
\delta_s\in\F\quad(s\in\mathcal Q\cap\mathcal S),
\]
such that
\begin{equation}\label{eq:left-two-one-relation}
	\lambda C_{ij}+\mu C_{kj}
	+
	\sum_{s\in\mathcal Q\cap\mathcal S}\delta_s\mathbf e_s=0.
\end{equation}
The left vertex $p$ does not belong to the left vertex set of
$\mathcal Q$.  Consequently, neither $pj$ nor $pq$, for
$q\in R_0$, belongs to $\mathcal Q\cap\mathcal S$.  Fix
$q\in R_0$.  In the relation \eqref{eq:left-two-one-relation}, the
coefficient of the basis vector $\mathbf e_{pj}$ is therefore
\begin{equation*}
	\lambda\alpha_p^{ij}+\mu\alpha_p^{kj}=0,
\end{equation*}
and the coefficient of the basis vector $\mathbf e_{pq}$ is
\begin{equation*}
	\lambda\gamma_{pq}^{ij}+\mu\gamma_{pq}^{kj}=0.
\end{equation*}
Since $\lambda$ and $\mu$ are non-zero, eliminating them gives~\eqref{eq:left-proportionality}. This holds for every choice of $p,i,k,j,q$.

Similarly, let $q\in R_0$, let $i\in L_1$, and let
$j,\ell\in R_1$ be distinct. Evaluating the circuit
\begin{equation}\label{eq:right-two-one}
	(L_0\cup\{i\})
	\times\bigl((R_0\setminus\{q\})\cup\{j,\ell\}\bigr)
\end{equation}
yields~\eqref{eq:right-proportionality}.
\end{proof}	
\begin{proposition}[Normalized star-basis form]
	\label{prop:normalized-star-form}
	Let $M$ be an $\F$-representable abstract $(a,b)$-birigidity matroid on
	$L\times R$. Then, there are non-zero parameters
	\[
	x_{i,p}\in\F^{\times}
	\quad(i\in L_1,\ p\in L_0)
	\]
	and
	\[
	y_{j,q}\in\F^{\times}
	\quad(j\in R_1,\ q\in R_0)
	\]
	satisfying
	\begin{equation}\label{eq:first-coordinates-one}
		x_{i,1}=1\quad(i\in L_1),
		\qquad
		y_{j,1}=1\quad(j\in R_1),
	\end{equation}
	such that $M$ has a star-basis representation with respect to
	$\mathcal S$ whose column indexed by $s\in\mathcal S$ is
	$\mathbf e_s$, and whose column $C_{ij}$ indexed by $ij$, for $i\in L_1$ and
	$j\in R_1$, is
	\begin{equation}\label{eq:normalized-star-column}
		C_{ij}
		=
		\sum_{p\in L_0}\sum_{q\in R_0}x_{i,p}y_{j,q}\mathbf e_{pq}
		+
		\sum_{p\in L_0}x_{i,p}\mathbf e_{pj}
		+
		\sum_{q\in R_0}y_{j,q}\mathbf e_{iq}.
	\end{equation}
\end{proposition}
For an illustration, a normalized star-basis matrix for an abstract birigidity matroid, with $(a,b,m,n)=(2,2,5,5)$, omitting the identity matrix block in the basis columns. To simplify notation, for $i=3,4,5$, $x_i$ and $y_i$ stand for $x_{i,2}$ and $y_{i,2}$ (and note that $x_{i,1}=y_{i,1}=1$).
\[
\begin{array}{c|ccccccccc}
	&C_{33}&C_{34}&C_{35}&C_{43}&C_{44}&C_{45}
	&C_{53}&C_{54}&C_{55}\\ \hline
	11&1&1&1&1&1&1&1&1&1\\
	12&y_3&y_4&y_5&y_3&y_4&y_5&y_3&y_4&y_5\\
	21&x_3&x_3&x_3&x_4&x_4&x_4&x_5&x_5&x_5\\
	22&x_3y_3&x_3y_4&x_3y_5&x_4y_3&x_4y_4&x_4y_5
	&x_5y_3&x_5y_4&x_5y_5\\
	13&1&0&0&1&0&0&1&0&0\\
	14&0&1&0&0&1&0&0&1&0\\
	15&0&0&1&0&0&1&0&0&1\\
	23&x_3&0&0&x_4&0&0&x_5&0&0\\
	24&0&x_3&0&0&x_4&0&0&x_5&0\\
	25&0&0&x_3&0&0&x_4&0&0&x_5\\
	31&1&1&1&0&0&0&0&0&0\\
	32&y_3&y_4&y_5&0&0&0&0&0&0\\
	41&0&0&0&1&1&1&0&0&0\\
	42&0&0&0&y_3&y_4&y_5&0&0&0\\
	51&0&0&0&0&0&0&1&1&1\\
	52&0&0&0&0&0&0&y_3&y_4&y_5
\end{array}
\]
\begin{proof}
If $L_1$ or $R_1$ has one element, an assertion that a quantity is
independent of the corresponding index is understood automatically.

In~\eqref{eq:initial-column}, let us scale\footnote{Since Lemma~\ref{lem:proportionality} holds for any star-basis representation, we may, in a slight abuse of notation, assume that after each row or column scaling~\eqref{eq:left-proportionality} and~\eqref{eq:right-proportionality} still hold with the `new' parameter values.} each non-basis column $ij$ (since $\gamma_{11}^{ij}\neq 0$) so that
\begin{equation}\label{eq:gamma11-one}
\gamma_{11}^{ij}=1
\qquad(i\in L_1,\ j\in R_1).
\end{equation}
Whenever a row is scaled below, scale the corresponding basis column
by the inverse scalar, so that every basis column remains the
corresponding column $\mathbf e_s$.

Consider \eqref{eq:left-proportionality} with $p=q=1$.  In view of
\eqref{eq:gamma11-one}, the coefficient $\alpha_1^{ij}$ does not depend on $i$.  Hence, we may scale row $1j$ so that
\begin{equation}\label{eq:alpha1-one}
\alpha_1^{ij}=1
\qquad(i\in L_1,\ j\in R_1).
\end{equation}
Similarly, consider \eqref{eq:right-proportionality} with $p=q=1$.  The coefficient
$\beta_1^{ij}$ does not depend on $j$. Hence, we may scale row $i1$ so that
\begin{equation}\label{eq:beta1-one}
\beta_1^{ij}=1
\qquad(i\in L_1,\ j\in R_1).
\end{equation}
For each $i\in L_1$ and $p\in L_0$, define
$$
x_{i,p}=\gamma_{p1}^{ij},
$$
where $j\in R_1$ is arbitrary.  This is well defined: by
\eqref{eq:right-proportionality} with $q=1$, together with
\eqref{eq:beta1-one}, $\gamma_{p1}^{ij}$ does not
depend on $j$.  Equation \eqref{eq:gamma11-one} gives $x_{i,1}=1$.
Similarly, for $j\in R_1$ and $q\in R_0$, define
$$
y_{j,q}=\gamma_{1q}^{ij},
$$
where $i\in L_1$ is arbitrary.  By
\eqref{eq:left-proportionality} with $p=1$, together with
\eqref{eq:alpha1-one}, this does not depend on $i$. We have
$y_{j,1}=1$.

For fixed $p\in L_0$ and $j\in R_1$, identity
\eqref{eq:left-proportionality} with $q=1$ gives that
\[
\frac{x_{i,p}}{\alpha_p^{ij}}
\]
is independent of $i\in L_1$. Scaling each row $pj$ by this non-zero scalar, we obtain
\begin{equation}\label{eq:alpha-equals-x}
\alpha_p^{ij}=x_{i,p}
\qquad(i\in L_1,\ j\in R_1,\ p\in L_0).
\end{equation}
Similarly, for fixed $i\in L_1$ and $q\in R_0$, equation
\eqref{eq:right-proportionality} with $p=1$ shows that row $iq$ can be
scaled so that
\begin{equation}\label{eq:beta-equals-y}
\beta_q^{ij}=y_{j,q}
\qquad(i\in L_1,\ j\in R_1,\ q\in R_0).
\end{equation}
Now, fix $i_0\in L_1$ and $j_0\in R_1$.  By
\eqref{eq:left-proportionality} and \eqref{eq:alpha-equals-x}, for
$i\in L_1$, $j\in R_1$, $p\in L_0$, and $q\in R_0$,
\begin{equation*}
\frac{\gamma_{pq}^{ij}}{x_{i,p}}
=
\frac{\gamma_{pq}^{i_0j}}{x_{i_0,p}}.
\end{equation*}
By \eqref{eq:right-proportionality} and \eqref{eq:beta-equals-y},
\begin{equation*}
\frac{\gamma_{pq}^{i_0j}}{y_{j,q}}
=
\frac{\gamma_{pq}^{i_0j_0}}{y_{j_0,q}}.
\end{equation*}
Consequently,
\begin{equation}\label{eq:gamma-kappa}
\gamma_{pq}^{ij}=\kappa_{pq}x_{i,p}y_{j,q},
\end{equation}
where
\begin{equation*}
\kappa_{pq}
=
\frac{\gamma_{pq}^{i_0j_0}}
{x_{i_0,p}y_{j_0,q}}
\in\F^{\times}.
\end{equation*}
The definitions of $x_{i,p}$ and $y_{j,q}$ imply
\[
\kappa_{p1}=1\quad(p\in L_0),
\qquad
\kappa_{1q}=1\quad(q\in R_0).
\]
Scale row $pq$ by $\kappa_{pq}^{-1}$ for every $p\in L_0$ and
$q\in R_0$.  Equations \eqref{eq:alpha-equals-x},
\eqref{eq:beta-equals-y}, and \eqref{eq:gamma-kappa} now give
\eqref{eq:normalized-star-column}.
\end{proof}
\begin{remark}\label{rem:minimal-hypotheses}
The above proof used the circuit condition only for the copies
\[
(L_0\cup\{i\})\times(R_0\cup\{j\})
\qquad(i\in L_1,\ j\in R_1)
\]
in order to deduce that all coefficients in
\eqref{eq:initial-column} are non-zero. For the copies in
\eqref{eq:left-two-one} and \eqref{eq:right-two-one}, on a closer inspection, dependence is sufficient. No other copies of $K_{a+1,b+1}$ were used in the proof. 
\end{remark}

\subsection{Conversion to the standard representation}
We next show that the normalized star-basis representation can be converted to an explicit instance of the standard birigidity matrix.

For every $p\in L_0$, let $\mathbf f_p\in\F^a$ be the vector whose
$p$th coordinate is $1$ and whose other coordinates are $0$.  For every
$q\in R_0$, let $\mathbf g_q\in\F^b$ be the vector whose $q$th
coordinate is $1$ and whose other coordinates are $0$.

\begin{lemma}[Reverse conversion]
\label{lem:reverse-conversion}
Suppose 	\[
x_{i,p}\in\F^{\times}
\quad(i\in L_1,\ p\in L_0)
\]
and
\[
y_{j,q}\in\F^{\times}
\quad(j\in R_1,\ q\in R_0)
\] satisfy
\eqref{eq:first-coordinates-one}. Suppose $M$ is a matroid on $L\times R$ admitting a star-basis representation, where each non-basis column is expressed by~\eqref{eq:normalized-star-column}.
 Define
\[
u_p=\mathbf f_p\in\F^a
\qquad(p\in L_0),
\]
\[
u_i=(-x_{i,1},\ldots,-x_{i,a})\in\F^a
\qquad(i\in L_1),
\]
\[
v_q=\mathbf g_q\in\F^b
\qquad(q\in R_0),
\]
and
\[
v_j=(-y_{j,1},\ldots,-y_{j,b})\in\F^b
\qquad(j\in R_1).
\]
Put
\[
u=(u_i)_{i\in L},\qquad v=(v_j)_{j\in R}.
\]
Then $M$ is the column matroid of $\mathcal R_{a,b}(u,v)$.
\end{lemma}
\begin{proof}
Let
$
H=\mathcal R_{a,b}(u,v)
$,
and let \(h_{ij}\) denote the column of \(H\) indexed by \(ij\).
We first prove that the columns indexed by \(\mathcal S\) are
independent. To this end, suppose that
$$
\sum_{\substack{p\in L_0\\ j\in R_1}}
\lambda_{pj}h_{pj}
+
\sum_{\substack{i\in L_1\\ q\in R_0}}
\lambda_{iq}h_{iq}
+
\sum_{\substack{p\in L_0\\ q\in R_0}}
\lambda_{pq}h_{pq}
=0.
$$
Fix \(j\in R_1\) and \(r\in L_0\). Evaluating coefficients in row \(V_{r,j}\), the only columns that can be non-zero in this row are \(h_{pj}\), \(p\in L_0\). Moreover, the entry of \(h_{pj}\) in \(V_{r,j}\) is $1$ if $p=r$, and $0$ otherwise. Hence,
\(\lambda_{rj}=0\), and since $r$ was arbitrary, we obtain
$$
\lambda_{pj}=0
\qquad(p\in L_0,\ j\in R_1),
$$
and, symmetrically,
$$
\lambda_{iq}=0
\qquad(i\in L_1,\ q\in R_0).
$$
Only the coefficients \(\lambda_{pq}\), with \(p\in L_0\) and
\(q\in R_0\), remain. Fix \(p\in L_0\) and \(s\in R_0\). In row
\(U_{p,s}\), the entry of \(h_{pq}\) is $1$ if $q=s$ and $0$ otherwise.
Consequently, \(\lambda_{ps}=0\). It follows that all the coefficients are zero.
Therefore $\{h_s:s\in\mathcal S\}$ is an independent set.

We next determine the coordinates of each non-basis column relative to
these columns. Fix \(i\in L_1\) and \(j\in R_1\), and put
$$
D_{ij}
=
\sum_{p\in L_0}\sum_{q\in R_0}
x_{i,p}y_{j,q}h_{pq}
+
\sum_{p\in L_0}x_{i,p}h_{pj}
+
\sum_{q\in R_0}y_{j,q}h_{iq}.
$$
We claim that
\begin{equation}\label{eq:standard-star-relation}
	D_{ij}=-h_{ij}.
\end{equation}
We verify this in every row of \(H\). Consider a row \(U_{\ell,s}\), where \(\ell\in L\) and
\(s\in [b]\) (the case of rows \(V_{r,t}\), where \(r\in [a]\) and \(t\in R\) is symmetric).
	
If \(\ell=i\), only the third sum in \(D_{ij}\) contributes. Its entry
in this row is
$$
\sum_{q\in R_0}y_{j,q}(v_q)_s
=
\sum_{q\in R_0}y_{j,q}(\mathbf g_q)_s
=
y_{j,s}.
$$
On the other hand, the entry of \(-h_{ij}\) is
$$
-(v_j)_s=y_{j,s}.
$$
If \(\ell=p\in L_0\), the contribution of the first sum is
$$
\sum_{q\in R_0}x_{i,p}y_{j,q}(v_q)_s
=
x_{i,p}y_{j,s},
$$
whereas the contribution of the second sum is
$$
x_{i,p}(v_j)_s=-x_{i,p}y_{j,s}.
$$
These contributions cancel, and the third sum has zero entry in this
row. Thus the entry of \(D_{ij}\) is zero, as is the entry of
\(-h_{ij}\).

Finally, if \(\ell\in L_1\setminus\{i\}\), none of the columns
occurring in \(D_{ij}\) is non-zero in row \(U_{\ell,s}\). Both sides of
\eqref{eq:standard-star-relation} therefore have zero entry in this
row. This proves \eqref{eq:standard-star-relation}.

Thus the coordinates of $-h_{ij}$ with respect to the basis
$\{h_s:s\in\mathcal S\}$ are exactly the coordinates in
\eqref{eq:normalized-star-column}.  Multiplying each non-basis column
$h_{ij}$ by $-1$ does not change the column matroid.  The two matrices
therefore represent the same matroid.
\end{proof}
\begin{corollary}\label{cor:every-instance}
Let $M$ be an $\F$-representable abstract $(a,b)$-birigidity matroid on
$L\times R$.  Then there are vectors
\[
u_i\in\F^a\quad(i\in L),
\qquad
v_j\in\F^b\quad(j\in R)
\]
such that, writing
\[
u=(u_i)_{i\in L},\qquad v=(v_j)_{j\in R},
\]
the matroid $M$ is the column matroid of $\mathcal R_{a,b}(u,v)$.
\end{corollary}
\begin{proof}
Apply Proposition~\ref{prop:normalized-star-form}, and then apply
Lemma~\ref{lem:reverse-conversion}.
\end{proof}

\subsection{Putting it all together}
\begin{proof}[Proof of Theorem~\ref{thm:global-maximality}]
Let $M$ be an abstract $(a,b)$-birigidity matroid
on $L\times R$, representable over a field $\F$ of characteristic $p$. By Corollary~\ref{cor:every-instance}, $M$ admits a standard $(a,b)$-birigidity matrix representation over $\F$. By Proposition~\ref{prop:generic-is-abstract}(ii) and (iii) we therefore have 
$$M\preceq \BR_{a,b}^{(p)}(m,n)\preceq \BR_{a,b}(m,n).
$$
By Proposition~\ref{prop:generic-is-abstract}(i), the matroid
$\BR_{a,b}(m,n)$ is itself a linearly representable abstract
$(a,b)$-birigidity matroid. 

Hence, $\BR_{a,b}(m,n)$ is the freest linearly representable abstract
$(a,b)$-birigidity matroid.
\end{proof}

\begin{proof}[Proof of Corollary~\ref{cor:tensor-freest}]
The tensor--birigidity duality of Brakensiek, Dhar, Gao, Gopi, and
Larson~\cite{Bra24}, together with its abstract formulation in~\cite{CJJT25}, identifies linearly representable tensor products of the
two uniform matroids with the duals of linearly representable abstract
$(a,b)$-birigidity matroids.  Under this identification, the generic
tensor product corresponds to $\BR_{a,b}(m,n)$.  The conclusion is
therefore the dual form of Theorem~\ref{thm:global-maximality}.
\end{proof}

\begin{proof}[Proof of Corollary~\ref{cor:potential-recognition}]
For arbitrary integers $a,b\geq 2$, $m\geq a+1$, and
$n\geq b+1$, let $P_{a,b}(m,n)$ denote the
corresponding potential matroid. It was shown in~\cite{Bra25} that $P_{a,b}(m,n)$ is an
abstract $(a,b)$-birigidity matroid and that
\[
\BR_{a,b}(m,n)
\preceq P_{a,b}(m,n).
\]
By assumption, each $P_{a,b}(m,n)$ is linearly representable.
Theorem~\ref{thm:global-maximality} therefore gives
\[
P_{a,b}(m,n)
\preceq\BR_{a,b}(m,n), 
\]
and thus
\[
P_{a,b}(m,n)
=\BR_{a,b}(m,n)
\]
for every admissible quadruple $(m,n,a,b)$.

These equalities supply every instance of the potential--birigidity
equality used in the proof of Lemma~5.18 of~\cite{Bra25}, including
both the equality at the original parameters $(m,n,a,b)$ and the
equality at the auxiliary parameters obtained after scaling and
transposition. Hence the correctness proof of Algorithm~1
in~\cite{Bra25} applies. Since the runtime analysis of that algorithm
is unconditional, for every fixed $a,b\geq2$ it decides independence
in $\BR_{a,b}(m,n)$ deterministically in time polynomial in $m$ and
$n$.

Put $E=[m]\times[n]$, and let
\[
B=\BR_{a,b}(m,n),
\qquad
T=U_{m-a,m}\otimes U_{n-b,n}.
\]
By the tensor--birigidity duality~\cite{Bra24}, $T=B^*$. The
independence oracle obtained above for $B$ yields a rank oracle for
$B$ in polynomial time by the greedy algorithm. For every
$X\subseteq E$, the dual rank formula gives
\[
r_T(X)=|X|-r_B(E)+r_B(E\setminus X).
\]
Thus \(r_T(X)\) can be computed in polynomial time. Since $X$ is
independent in $T$ if and only if \(r_T(X)=|X|\), we obtain the
claimed deterministic polynomial-time independence algorithm for the
generic tensor product.	
\end{proof}

\section{Two-dimensional rigidity}\label{sec:two-rigidity}
In this section we establish analogous results for $2$-dimensional abstract rigidity matroids, culminating in the proof of Theorem~\ref{thm:rigidity-realization}.

\subsection{Abstract $2$-rigidity matroids and two representations}
Let \(n\geq 4\), 
and write \(ij\) for \(\{i,j\}\in\binom{[n]}{2}\).  For a set \(T\subseteq\{1,\ldots,n\}\), write
\[
K_T=\{ij:i,j\in T,\ i<j\}
\]
for the edge set of the complete graph on \(T\).

We recall the definition of an abstract $2$-rigidity matroid.
\begin{definition}
An \emph{abstract $2$-rigidity matroid} on $\binom{[n]}{2}$ is a matroid $M$ such that
\begin{enumerate}
	\item $\rank M =2n-3$;
	\item For every $T\in \binom{[n]}{4}$, $K_T$ is a circuit.
\end{enumerate}	
\end{definition}
\begin{definition}[The infinitesimal rigidity matrix, transposed]
Let $\F$ be a field, and let 
\[
p_i=(a_i,b_i)\in \F^2, \ \ i\in[n].
\]
The (transposed) two-dimensional infinitesimal rigidity matrix 
$$R(p_1,\ldots,p_n)
$$
has rows
\[
A_1,B_1,\ldots,A_n,B_n
\]
and one column for each edge \(ij\in \binom{[n]}{2}\).  The column \(ij\), where $i<j$, has entries $a_i-a_j, b_i-b_j, a_j-a_i,\ b_j-b_i$ in rows $A_i,B_i,A_j,B_j$ respectively
and all other entries are zero.
\end{definition}
\begin{definition}[The generic two-dimensional rigidity matroid]
	Fix a characteristic $p$, where $p=0$ or $p$ is prime. Let
	$a_1,b_1,\ldots,a_n,b_n$ be algebraically independent indeterminates
	over $k_p$. The generic two-dimensional rigidity matroid in
	characteristic $p$, denoted by $R_{2,n}^{(p)}$, is the column matroid of
	$R\bigl((a_1,b_1),\ldots,(a_n,b_n)\bigr)$ over the field $k_p(a_1,b_1,\ldots,a_n,b_n)$.
\end{definition}
The characteristic-free form of the Geiringer--Laman
theorem~\cite{PG27,Lam70,Whi96} identifies $R_{2,n}^{(p)}$ with the
$(2,3)$-sparsity matroid for every $p$. We denote this common matroid by
$R_{2,n}$. These matroids are invariant under permutations of $[n]$ and
compatible under restrictions to subsets of the vertex set. The generic
two-dimensional rigidity family, denoted by $\mathcal R_2$, consists of
the finite graphs $G$ whose edge sets are independent in $R_{2,v}$ after
identifying $V(G)$ with $[v]$, where $v=|V(G)|$. Suppose $\F$ is a field of characteristic $p$, and let
$a_1,b_1,\ldots,a_n,b_n$ be algebraically independent indeterminates
over $\F$. By Lemma~\ref{lem:scalar-extension}, the matrix
\[
R\bigl((a_1,b_1),\ldots,(a_n,b_n)\bigr)
\]
represents $R_{2,n}^{(p)}=R_{2,n}$ over
$\F(a_1,b_1,\ldots,a_n,b_n)$.

Let $M$ be an abstract $2$-rigidity matroid on $n$ vertices. Let 
$I=\{3,4,\ldots,n\}$, and  
\[
B=\{12\}\cup\{1i:i\in I\}\cup\{2i:i\in I\}.
\]
For any distinct $i,j\in I$, since $K_{\{1,2,i,j\}}$ is an $M$-circuit whose edges other than $ij$ belong to $B$, we have $ij\in \operatorname{cl}_M(B)$. Hence, $B$ spans $M$. Since $|B|=2n-3$, $B$ is a basis of $M$. Suppose $M$ is representable over a field $\F$. For each $ij\in B$, let $e_{ij}\in \F^B$ be the vector having $1$ in coordinate $ij$ and $0$ otherwise. Since $B$ is a basis and, for any distinct $i,j\in I$, $K_{\{1,2,i,j\}}$ is the fundamental circuit of $ij$ with respect to it, we may represent $M$ by a $(2n-3)\times \binom{n}{2}$-matrix, whose rows are indexed by $B$ and columns by $\binom{[n]}{2}$, such that the columns indexed by $B$ form an identity block, and the remaining columns $C_{ij}, i,j\in I, i<j$, satisfy (after normalization)
\begin{equation}\label{eq:alpha-beta-column}
	C_{ij}=e_{12}+\alpha_{ij}e_{1i}+\alpha_{ji}e_{1j}
	+\beta_{ij}e_{2i}+\beta_{ji}e_{2j},
\end{equation}
with non-zero coefficients. We call this a \emph{star-basis} representation.
\begin{theorem}[Star-basis normal form]\label{thm:abstract-normal-form}
Let \(M\) be an abstract $2$-rigidity matroid on \(\binom{[n]}{2}\), representable over \(\F\).  
Then \(M\) admits a representation by a star-basis matrix with non-basis columns
\begin{equation}\label{eq:star-column}
	C_{ij}=e_{12}+\frac{e_{1i}-e_{1j}}{x_i-x_j}
	+\frac{e_{2i}-e_{2j}}{y_j-y_i},
\end{equation}
for some parameters  \(x_i,y_i\in \F\), \(i\in I\), satisfying
\[
x_i\neq x_j,
\qquad
 y_i\neq y_j
\qquad (i,j\in I,\ i\neq j)
\]
and
\[
x_i\neq y_i
\qquad (i\in I).
\]
\end{theorem}
Given two vectors $x=(x_i)_{i\in I}, y=(y_i)_{i\in I} \in \F^I$, with the $x_i$ distinct, and the $y_i$ distinct, denote by \(S(x,y)\) the star-basis representation given by~\eqref{eq:star-column}.
\begin{proof}
Choose a star-basis representation over $\F$ satisfying~\eqref{eq:alpha-beta-column}. For $n=4$ a direct scaling argument shows that~\eqref{eq:star-column} can be assumed for the unique non-basis column $C_{34}$. So, let us from now assume that $n\geq 5$, i.e. $|I|\geq 3$.  

Let \(i,j,k\in I\) be distinct.  Since
\[
K_{\{2,i,j,k\}}=\{2i,2j,2k,ij,ik,jk\}
\]
is dependent, and since \(2i,2j,2k\) are basis elements, the columns of
\begin{equation*}
	\begin{array}{c|ccc|ccc}
		&2i&2j&2k&ij&ik&jk\\ \hline
		12&0&0&0&1&1&1\\
		1i&0&0&0&\alpha_{ij}&\alpha_{ik}&0\\
		1j&0&0&0&\alpha_{ji}&0&\alpha_{jk}\\
		1k&0&0&0&0&\alpha_{ki}&\alpha_{kj}\\ 
		2i&1&0&0&\beta_{ij}&\beta_{ik}&0\\
		2j&0&1&0&\beta_{ji}&0&\beta_{jk}\\
		2k&0&0&1&0&\beta_{ki}&\beta_{kj}
	\end{array}
\end{equation*}
are dependent. Consequently, the columns of 
\begin{equation}\label{eq:alpha-matrix}
\begin{array}{c|ccc}
& ij & ik & jk \\ \hline
12 & 1 & 1 & 1 \\
1i & \alpha_{ij} & \alpha_{ik} & 0 \\
1j & \alpha_{ji} & 0 & \alpha_{jk} \\
1k & 0 & \alpha_{ki} & \alpha_{kj}
\end{array}
\end{equation}
are dependent. The determinant of the submatrix using rows \(1i,1j,1k\) gives
\begin{equation}\label{eq:alpha-product}
\alpha_{ij}\alpha_{jk}\alpha_{ki}
+
\alpha_{ik}\alpha_{ji}\alpha_{kj}=0.
\end{equation}
Define non-zero scalars
\[
r_3=1,
\qquad
r_i=-\frac{\alpha_{3i}}{\alpha_{i3}}
\quad (i\in I,\ i\neq 3).
\]
For every \(i\in I\), scale row \(1i\) by \(r_i\), and scale the basis column \(1i\) by \(r_i^{-1}\), so that the submatrix on the basis columns remains the identity.  For every ordered pair \((i,j)\) of distinct elements of \(I\), put
\[
\lambda_{ij}=r_i\alpha_{ij}.
\]
For pairs involving \(3\), the definition gives \(\lambda_{i3}=-\lambda_{3i}\).  For \(i,j\in I\setminus\{3\}\) with \(i\neq j\), equation \eqref{eq:alpha-product} applied to the triple \(3,i,j\) gives
\[
r_i\alpha_{ij}+r_j\alpha_{ji}=0,
\]
so again
\begin{equation}\label{eq:lambda-skew}
\lambda_{ji}=-\lambda_{ij}
\qquad (i,j\in I,\ i\neq j).
\end{equation}
The minor of \eqref{eq:alpha-matrix} using rows \(12,1i,1j\), after the above row scalings, is
\[
\det
\begin{pmatrix}
1&1&1\\
\lambda_{ij}&\lambda_{ik}&0\\
-\lambda_{ij}&0&\lambda_{jk}
\end{pmatrix}=0.
\]
Expanding and dividing by the non-zero product \(\lambda_{ij}\lambda_{ik}\lambda_{jk}\), one obtains
\begin{equation}\label{eq:lambda-additive}
\frac{1}{\lambda_{ij}}-\frac{1}{\lambda_{ik}}+
\frac{1}{\lambda_{jk}}=0
\qquad (i,j,k\in I\text{ distinct}).
\end{equation}
Define
\begin{equation}\label{eq:x-def}
x_3=0,
\qquad
x_i=-\frac{1}{\lambda_{3i}}
\quad (i\in I,\ i\neq 3).
\end{equation}
For every \(j\in I\setminus\{3\}\), equations \eqref{eq:lambda-skew} and \eqref{eq:x-def} give
\[
\frac{1}{\lambda_{3j}}=x_3-x_j,
\qquad
\frac{1}{\lambda_{j3}}=x_j-x_3.
\]
For \(i,j\in I\setminus\{3\}\) with \(i\neq j\), equation \eqref{eq:lambda-additive} applied to the triple \(3,i,j\) gives
\[
\frac{1}{\lambda_{ij}}=x_i-x_j.
\]
Hence, \(x_i\neq x_j\), and
\begin{equation*}
\lambda_{ij}=\frac{1}{x_i-x_j}
\qquad (i,j\in I,\ i\neq j).
\end{equation*}
The \(\beta\)-entries are treated in the same way, using the dependent sets
\[
K_{\{1,i,j,k\}}=\{1i,1j,1k,ij,ik,jk\}
\qquad (i,j,k\in I\text{ distinct}).
\]
Define non-zero scalars
\[
s_3=1,
\qquad
s_i=-\frac{\beta_{3i}}{\beta_{i3}}
\quad (i\in I,\ i\neq 3).
\]
For every \(i\in I\), scale row \(2i\) by \(s_i\), and scale the basis column \(2i\) by \(s_i^{-1}\).  For every ordered pair \((i,j)\) of distinct elements of \(I\), put
\[
\mu_{ij}=s_i\beta_{ij}.
\]
The identical determinant argument gives
\[
\mu_{ji}=-\mu_{ij}
\qquad (i,j\in I,\ i\neq j)
\]
and
\[
\frac{1}{\mu_{ij}}-\frac{1}{\mu_{ik}}+
\frac{1}{\mu_{jk}}=0
\qquad (i,j,k\in I\text{ distinct}).
\]
Define
\begin{equation*}
y_3=0,
\qquad
 y_i=\frac{1}{\mu_{3i}}
\quad (i\in I,\ i\neq 3).
\end{equation*}
So, \(y_i\neq y_j\) and
\begin{equation*}
\mu_{ij}=\frac{1}{y_j-y_i}
\qquad (i,j\in I,\ i\neq j).
\end{equation*}
After these row scalings and compensating basis-column scalings, equation \eqref{eq:alpha-beta-column} has become exactly
\[
C_{ij}=e_{12}+\frac{e_{1i}-e_{1j}}{x_i-x_j}
        +\frac{e_{2i}-e_{2j}}{y_j-y_i}
\qquad (i,j\in I,\ i<j).
\]
This proves the asserted star-basis form.

It remains to impose \(x_i-y_i\neq 0\). Suppose \(a\in\F^\times\) and \(c\in\F\), and put \(y_i'=ay_i+c\) for every \(i\in I\). The matrix \(S(x,y')\) is obtained from \(S(x,y)\) by multiplying every row indexed by \(2i\), \(i\in I\), by \(a^{-1}\), and then multiplying every basis column indexed by \(2i\) by \(a\). Hence \(S(x,y')\) represents the same matroid as \(S(x,y)\). It therefore suffices to choose \(a\in\F^\times\) and \(c\in\F\) such that
$$
c\notin\{x_i-ay_i:i\in I\}.
$$
If \(\F\) is infinite, we may take \(a=1\) and choose \(c\) outside this finite set. Suppose that \(\F\) is finite, and put \(q=|\F|\). Since the \(y_i\) are distinct, we have \(|I|\leq q\). If \(|I|<q\), we may again take \(a=1\), and choose an element $c$ not from the above set.
 If \(|I|=q\), choose distinct \(r,s\in I\) and put
$$
a=\frac{x_r-x_s}{y_r-y_s}\in\F^\times.
$$
Then \(x_r-ay_r=x_s-ay_s\), so the set \(\{x_i-ay_i:i\in I\}\) has at most \(q-1\) elements. We may therefore choose \(c\) outside this set. In every case, the resulting parameters satisfy
$$
x_i-y_i'=x_i-ay_i-c\neq0
\qquad(i\in I).
$$
Since \(a\neq0\), the parameters \(y_i'\) remain distinct. Replacing \(y\) by \(y'\) completes the proof.
\end{proof}

\begin{remark}
The proof uses the circuit property only for the sets \(K_{\{1,2,i,j\}}\).  For these sets, being circuits forces all five coordinates of \(C_{ij}\) in the rows \(12,1i,1j,2i,2j\) to be non-zero.  For the sets \(K_{\{1,i,j,k\}}\) and \(K_{\{2,i,j,k\}}\), the argument uses only dependence. The proof uses no other copies of $K_4$.
\end{remark}

Choosing $x$ and $y$ generically gives the freest among all $S(x,y)$-representable matroids in the given characteristic. Doing so in characteristic $0$ then gives the freest $S(x,y)$-representable matroid across all characteristics. Unsurprisingly, this is the generic rigidity matroid $R_{2,n}$. Below we make this more explicit by showing that every $\F$-representable abstract $2$-rigidity matroid admits an $R(p_1,\dots,p_n)$ representation over $\F$.
\subsection{The reverse conversion}

We show that the column matroid of \(S(x,y)\) can always be realized by an infinitesimal rigidity matrix over the same field.

\begin{proposition}\label{prop:rigidity-to-star}
Let \(x_i,y_i\in \F\), \(i\in I\), satisfy
\[
x_i\neq x_j,
\qquad
 y_i\neq y_j
\qquad
(i,j\in I,
 i\neq j),
\]
and
\[
x_i-y_i\neq 0
\qquad
(i\in I).
\]
Then there exist points \(p_1,\ldots,p_n\in \F^2\) such that the column matroid of \(S(x,y)\) is represented by
\[
R(p_1,\ldots,p_n).
\]
\end{proposition}

\begin{proof}
Define points
\begin{equation*}
p_1=(0,0),
\qquad
p_2=(1,0),
\qquad
p_i=\left(\frac{x_i}{x_i-y_i},\frac{1}{x_i-y_i}\right)
\quad (i\in I).
\end{equation*}
For every \(i\in I\), the entries in rows \(A_i,B_i\) of the column \(1i\) are
\[
\left(\frac{x_i}{x_i-y_i},\frac{1}{x_i-y_i}\right),
\]
and the entries in rows \(A_i,B_i\) of the column \(2i\) are
\[
\left(\frac{y_i}{x_i-y_i},\frac{1}{x_i-y_i}\right).
\]
The determinant of the corresponding \(2\times 2\) matrix is
\[
\frac{x_i}{(x_i-y_i)^2}-\frac{y_i}{(x_i-y_i)^2}
=\frac{1}{x_i-y_i}\neq 0.
\]
Thus any linear dependence among the columns in \(B\) has zero coefficients on \(1i\) and \(2i\) for every \(i\in I\).  The remaining coefficient on \(12\) is then also zero, so \(B\) is independent in \(R(p_1,\ldots,p_n)\).

Fix \(i,j\in I\) with \(i<j\).  Define the following linear combination of the columns of \(R(p_1,\ldots,p_n)\):
\begin{align*}
\rho_{ij}={}&\frac{(x_i-x_j)(y_j-y_i)}{(x_i-y_i)(x_j-y_j)}\,12
 +\frac{y_i-y_j}{x_j-y_j}\,1i
 +\frac{y_j-y_i}{x_i-y_i}\,1j  \\
&+\frac{x_j-x_i}{x_j-y_j}\,2i
 +\frac{x_i-x_j}{x_i-y_i}\,2j. 
\end{align*}
We claim that \(\rho_{ij}\) is equal to the column \(ij\).  In every row different from
\[
A_1,B_1,A_2,B_2,A_i,B_i,A_j,B_j,
\]
both columns have entry zero.  We therefore compare the remaining rows.

The entries of \(\rho_{ij}\) in rows \(A_1\) and \(B_1\) are
\[
-\frac{(x_i-x_j)(y_j-y_i)}{(x_i-y_i)(x_j-y_j)}
-\frac{x_i(y_i-y_j)}{(x_i-y_i)(x_j-y_j)}
-\frac{x_j(y_j-y_i)}{(x_i-y_i)(x_j-y_j)}=0
\]
and
\[
-\frac{y_i-y_j}{(x_i-y_i)(x_j-y_j)}
-\frac{y_j-y_i}{(x_i-y_i)(x_j-y_j)}=0.
\]
These are the entries of the column \(ij\) in rows \(A_1\) and \(B_1\).

The entries of \(\rho_{ij}\) in rows \(A_2\) and \(B_2\) are
\[
\frac{(x_i-x_j)(y_j-y_i)}{(x_i-y_i)(x_j-y_j)}
-\frac{y_i(x_j-x_i)}{(x_i-y_i)(x_j-y_j)}
-\frac{y_j(x_i-x_j)}{(x_i-y_i)(x_j-y_j)}=0
\]
and
\[
-\frac{x_j-x_i}{(x_i-y_i)(x_j-y_j)}
-\frac{x_i-x_j}{(x_i-y_i)(x_j-y_j)}=0.
\]
These are the entries of the column \(ij\) in rows \(A_2\) and \(B_2\).

The entries of \(\rho_{ij}\) in rows \(A_i\) and \(B_i\) are
\begin{align*}
\frac{x_i(y_i-y_j)+y_i(x_j-x_i)}{(x_i-y_i)(x_j-y_j)}
&=\frac{x_jy_i-x_iy_j}{(x_i-y_i)(x_j-y_j)}  \\
&=\frac{x_i}{x_i-y_i}-\frac{x_j}{x_j-y_j}
\end{align*}
and
\begin{align*}
\frac{(y_i-y_j)+(x_j-x_i)}{(x_i-y_i)(x_j-y_j)}
&=\frac{(x_j-y_j)-(x_i-y_i)}{(x_i-y_i)(x_j-y_j)} \\
&=\frac{1}{x_i-y_i}-\frac{1}{x_j-y_j}.
\end{align*}
These are precisely the entries of the column \(ij\) in rows \(A_i\) and \(B_i\), respectively.

The entries of \(\rho_{ij}\) in rows \(A_j\) and \(B_j\) are
\begin{align*}
\frac{x_j(y_j-y_i)+y_j(x_i-x_j)}{(x_i-y_i)(x_j-y_j)}
&=\frac{x_iy_j-x_jy_i}{(x_i-y_i)(x_j-y_j)}  \\
&=\frac{x_j}{x_j-y_j}-\frac{x_i}{x_i-y_i}
\end{align*}
and
\begin{align*}
\frac{(y_j-y_i)+(x_i-x_j)}{(x_i-y_i)(x_j-y_j)}
&=\frac{(x_i-y_i)-(x_j-y_j)}{(x_i-y_i)(x_j-y_j)} \\
&=\frac{1}{x_j-y_j}-\frac{1}{x_i-y_i}.
\end{align*}
These are precisely the entries of the column \(ij\) in rows \(A_j\) and \(B_j\), respectively.  Hence \(\rho_{ij}=ij\).  Therefore
\begin{align}\label{eq:unscaled-relation}
 ij={}&\frac{(x_i-x_j)(y_j-y_i)}{(x_i-y_i)(x_j-y_j)}\,12
 +\frac{y_i-y_j}{x_j-y_j}\,1i
 +\frac{y_j-y_i}{x_i-y_i}\,1j  \\
&+\frac{x_j-x_i}{x_j-y_j}\,2i
 +\frac{x_i-x_j}{x_i-y_i}\,2j. \nonumber
\end{align}
Now apply the non-zero column scalings
\begin{equation}\label{eq:column-scalings}
\widehat{12}=12,
\qquad
\widehat{1i}=-(x_i-y_i)\,1i,
\qquad
\widehat{2i}=-(x_i-y_i)\,2i
\qquad (i\in I),
\end{equation}
and
\begin{equation}\label{eq:nonbasis-column-scaling}
\widehat{ij}=\frac{(x_i-y_i)(x_j-y_j)}{(x_i-x_j)(y_j-y_i)}\,ij
\qquad (i,j\in I,
\ i<j).
\end{equation}
Multiplying \eqref{eq:unscaled-relation} by the scalar in \eqref{eq:nonbasis-column-scaling} and then using \eqref{eq:column-scalings}, the coefficient of \(\widehat{12}\) becomes \(1\), the coefficients of \(\widehat{1i}\) and \(\widehat{1j}\) become
\[
\frac{1}{x_i-x_j},
\qquad
\frac{1}{x_j-x_i},
\]
and the coefficients of \(\widehat{2i}\) and \(\widehat{2j}\) become
\[
\frac{1}{y_j-y_i},
\qquad
\frac{1}{y_i-y_j}.
\]
Thus the \(B\)-coordinate vector of every non-basis column is precisely the corresponding column of \(S(x,y)\).  Since non-zero column scalings and passage to coordinates in an independent column basis preserve the represented matroid, the column matroid of \(R(p_1,\ldots,p_n)\) is equal to the column matroid of \(S(x,y)\).
\end{proof}
Combining Theorem~\ref{thm:abstract-normal-form} with
Proposition~\ref{prop:rigidity-to-star} gives the claimed
rigidity-matrix realization.
\begin{proof}[Proof of Theorem~\ref{thm:rigidity-realization}]
Let $M$ be an abstract $2$-rigidity matroid, representable over a field $\F$.
Theorem~\ref{thm:abstract-normal-form} represents $M$ by a matrix $S(x,y)$ of the form
\eqref{eq:star-column}, where the parameters $x_i,y_i\in \F$, $i\in I$ satisfy
\[
x_i\neq x_j,
\qquad
 y_i\neq y_j
\qquad
(i,j\in I,
 i\neq j),
\]
and 
\[
x_i-y_i\neq 0
\qquad (i\in I).
\]
Proposition~\ref{prop:rigidity-to-star} applies to these parameters.  Hence there exist points \(p_1,\ldots,p_n\in \F^2\) such that \(R(p_1,\ldots,p_n)\) represents the column matroid of \(S(x,y)\).  Since \(S(x,y)\) represents \(M\), the matrix \(R(p_1,\ldots,p_n)\) represents \(M\).
\end{proof}

\section{Two-dimensional hyperconnectivity}\label{sec:hyperconnectivity}
\subsection{Hyperconnectivity matrices and star-basis form with $K_{3,3}$-circuits}

We retain the notation $\F$, $I$, $B$, and $S(x,y)$ from Section~\ref{sec:two-rigidity}. By
Theorem~\ref{thm:abstract-normal-form}, every $\F$-representable
abstract $2$-rigidity matroid admits a representation by a matrix
$S(x,y)$ of the form \eqref{eq:star-column}.

\begin{definition}[The two-dimensional hyperconnectivity matrix]\label{def:hyper-matrix}
Let
\[
q_i=(c_i,d_i)\in \F^2
\qquad (1\leq i\leq n).
\]
The transposed form of the usual two-dimensional hyperconnectivity matrix is denoted by
\[
H(q_1,\ldots,q_n).
\]
It has rows
\[
C_1,D_1,\ldots,C_n,D_n
\]
and one column \(h_{ij}\) for each edge \(ij\), \(i<j\).  The column \(h_{ij}\) has entries
\[
H_{C_i,ij}=c_j,
\qquad
H_{D_i,ij}=d_j,
\]
\[
H_{C_j,ij}=-c_i,
\qquad
H_{D_j,ij}=-d_i,
\]
and all other entries equal to zero.
\end{definition}
\begin{definition}[The generic two-dimensional hyperconnectivity matroid]
	Fix a characteristic $p$, where $p=0$ or $p$ is prime. Let
	$c_1,d_1,\ldots,c_n,d_n$ be algebraically independent indeterminates
	over $k_p$. The generic two-dimensional hyperconnectivity matroid in
	characteristic $p$, denoted by $H_{2,n}^{(p)}$, is the column matroid of
$H\bigl((c_1,d_1),\ldots,(c_n,d_n)\bigr)$
	over the field $k_p(c_1,d_1,\ldots,c_n,d_n)$.
\end{definition}
It was proved in~\cite[Corollary~1.13]{Tyo26} that
\begin{equation}\label{eq:Hcharindep}
H_{2,n}^{(p)}=H_{2,n}^{(0)}
\qquad\text{for every prime }p.
\end{equation}
We denote this common matroid by $H_{2,n}$. These matroids are invariant
under permutations of $[n]$ and compatible under restrictions to
subsets of the vertex set. The generic two-dimensional
hyperconnectivity family, denoted by $\mathcal H_2$, consists of the
finite graphs $G$ whose edge sets are independent in $H_{2,v}$ after
identifying $V(G)$ with $[v]$, where $v=|V(G)|$.

Suppose $\F$ is a field of characteristic $p$, and let
$c_1,d_1,\ldots,c_n,d_n$ be algebraically independent indeterminates
over $\F$. By Lemma~\ref{lem:scalar-extension}, the matrix
\[
H\bigl((c_1,d_1),\ldots,(c_n,d_n)\bigr)
\]
represents $H_{2,n}^{(p)}=H_{2,n}$ over
$\F(c_1,d_1,\ldots,c_n,d_n)$.

We now show that under the \(K_{3,3}\)-circuit condition we may assume \(y_i=1/x_i\) for all $i\in I$. Throughout this subsection, assume that $n\geq6$.
\begin{proposition}\label{prop:k33-fractional}
Let \(S(x,y)\) be a star-basis matrix defined in
\eqref{eq:star-column}.  Suppose further that in the column matroid of \(S(x,y)\) every copy of \(K_{3,3}\) is a circuit.  Then there are parameters \(X_i,Y_i\in\F^\times\), \(i\in I\), such that the column matroids of \(S(X,Y)\) and \(S(x,y)\) are the same, and
\begin{equation}\label{eq:Y-one-over-X}
Y_i=\frac{1}{X_i}
\qquad (i\in I).
\end{equation}
\end{proposition}

\begin{proof}
Let \(i,j,k,\ell\in I\) be distinct.  Consider the copy of \(K_{3,3}\) with bipartition
\[
\{1,2,i\}\sqcup\{j,k,\ell\}.
\]
Its edges are
\[
1j,1k,1\ell,
\qquad
2j,2k,2\ell,
\qquad
ij,ik,i\ell.
\]
The first six of these are basis edges. So, similarly to the argument in the proof of Theorem~\ref{thm:abstract-normal-form},
we obtain that the three columns
\[
\begin{pmatrix}
1\\[2pt]
\dfrac{1}{x_i-x_j}\\[6pt]
\dfrac{1}{y_j-y_i}
\end{pmatrix},
\qquad
\begin{pmatrix}
1\\[2pt]
\dfrac{1}{x_i-x_k}\\[6pt]
\dfrac{1}{y_k-y_i}
\end{pmatrix},
\qquad
\begin{pmatrix}
1\\[2pt]
\dfrac{1}{x_i-x_\ell}\\[6pt]
\dfrac{1}{y_\ell-y_i}
\end{pmatrix}
\]
are dependent.  In other words,
\begin{equation}\label{eq:k33-det}
\det
\begin{pmatrix}
1&1&1\\[4pt]
\dfrac{1}{x_i-x_j}&\dfrac{1}{x_i-x_k}&\dfrac{1}{x_i-x_\ell}\\[10pt]
\dfrac{1}{y_j-y_i}&\dfrac{1}{y_k-y_i}&\dfrac{1}{y_\ell-y_i}
\end{pmatrix}=0.
\end{equation}
Apply this with \(i=3\).  Since \(n\geq 6\), the set \(I\setminus\{3\}\) has at least three elements.  Put
\[
 u_t=\frac{1}{x_3-x_t},
 \qquad
 v_t=\frac{1}{y_t-y_3}
 \qquad (t\in I\setminus\{3\}).
\]
Equation \eqref{eq:k33-det}, for all triples in \(I\setminus\{3\}\), says exactly that the determinant
\[
\det
\begin{pmatrix}
1&1&1\\
u_j&u_k&u_\ell\\
v_j&v_k&v_\ell
\end{pmatrix}
\]
vanishes for all distinct \(j,k,\ell\in I\setminus\{3\}\).  

Subtracting the first column from the second and third gives
\[
\det
\begin{pmatrix}
	1&1&1\\
	u_j&u_k&u_\ell\\
	v_j&v_k&v_\ell
\end{pmatrix}
=
\det
\begin{pmatrix}
	u_k-u_j&u_\ell-u_j\\
	v_k-v_j&v_\ell-v_j
\end{pmatrix}.
\]
Hence the determinant is zero if and only if
\[
(u_k-u_j)(v_\ell-v_j)=(u_\ell-u_j)(v_k-v_j).
\]
Since the \(x_t\) are distinct, the \(u_t\) are distinct. Thus \(u_k-u_j\ne0\), and the last equation says that
\[
v_\ell-v_j=\frac{v_k-v_j}{u_k-u_j}(u_\ell-u_j).
\]
Therefore \((u_\ell,v_\ell)\) lies on the affine line through
\((u_j,v_j)\) and \((u_k,v_k)\). Fix some $j,k\in I\setminus \{3\}$, e.g. $j=4$ and $k=5$. Then all points $(u_t,v_t), t\in I\setminus \{3\}$ lie on a common affine line $v=\alpha u+\beta$ with
\[
\alpha=\frac{v_5-v_4}{u_5-u_4},
\qquad
\beta=v_4-\alpha u_4.
\]
Since the $y_i$ are distinct, we have $v_5\ne v_4$, and since the $x_i$ are distinct, we have $u_5\ne u_4$. Hence $\alpha\neq 0$.

Thus, we have
\begin{equation}\label{eq:line-equation}
\frac{1}{y_t-y_3}=\frac{\alpha}{x_3-x_t}+\beta
\qquad (t\in I\setminus\{3\}).
\end{equation}
Solving for \(y_t\) gives
\begin{equation}\label{eq:fractional-relation-alpha-beta}
y_t=y_3+\frac{x_3-x_t}{\alpha+\beta(x_3-x_t)}
\qquad (t\in I\setminus\{3\}), 
\end{equation}
where the denominator in \eqref{eq:fractional-relation-alpha-beta} is non-zero for each \(t\).

We next show that \(\beta\neq 0\).  Suppose, for a contradiction, that \(\beta=0\).  Then \eqref{eq:line-equation} gives
\[
y_t-y_3=\frac{x_3-x_t}{\alpha}
\qquad (t\in I\setminus\{3\}),
\]
and hence there are \(\lambda\in\F^\times\) and \(\mu\in\F\) such that
\begin{equation}\label{eq:affine-relation}
y_r=\lambda x_r+\mu
\qquad (r\in I).
\end{equation}
Choose $i,j,k,\ell\in I$ such that $i<j<k<\ell$. Suppose $a,b\in I$ and $a<b$. By \eqref{eq:star-column} and \eqref{eq:affine-relation}, we have
\[
(x_a-x_b)C_{ab}
=(x_a-x_b)e_{12}+e_{1a}-e_{1b}
-\frac{1}{\lambda}(e_{2a}-e_{2b}).
\]
Taking the alternating sum of this identity for
$(a,b)=(i,k),(i,\ell),(j,k),(j,\ell)$ gives
\[
(x_i-x_k)C_{ik}
-(x_i-x_\ell)C_{i\ell}
-(x_j-x_k)C_{jk}
+(x_j-x_\ell)C_{j\ell}=0.
\]
Since the $x_r$ are distinct, all four coefficients in this relation
are non-zero. Thus the four columns $C_{ik},C_{i\ell},C_{jk},C_{j\ell}$
are dependent. The corresponding four edges form a proper subset of
the copy of $K_{3,3}$ with bipartition
$\{1,i,j\}\sqcup\{2,k,\ell\}$, contradicting the assumption that this
copy is a circuit. Therefore $\beta\neq0$.

Define, for every \(r\in I\),
\begin{equation}\label{eq:X-def}
X_r=\alpha+\beta(x_3-x_r)
\end{equation}
and
\begin{equation}\label{eq:Y-def}
Y_r=\frac{1-\beta(y_r-y_3)}{\alpha}.
\end{equation}
For \(r=3\), equations \eqref{eq:X-def} and \eqref{eq:Y-def} give \(X_3=\alpha\) and \(Y_3=1/\alpha\).  For \(r\neq3\), we have \(X_r\neq0\) by the observation following
\eqref{eq:fractional-relation-alpha-beta}, and equation \eqref{eq:line-equation} gives
\[
y_r-y_3=\frac{x_3-x_r}{\alpha+\beta(x_3-x_r)}=\frac{x_3-x_r}{X_r}.
\]
Consequently
\[
Y_r=\frac{1-\beta(x_3-x_r)/X_r}{\alpha}
=\frac{X_r-\beta(x_3-x_r)}{\alpha X_r}
=\frac{1}{X_r}.
\]
Thus \eqref{eq:Y-one-over-X} holds for every \(r\in I\).  Also, since \(\beta\neq0\) and the \(x_r\) are distinct, the \(X_r\) are distinct.

It remains to observe that \(S(X,Y)\) represents the same matroid as \(S(x,y)\).  For all distinct \(i,j\in I\), we have
\[
X_i-X_j=-\beta(x_i-x_j)
\]
and
\[
Y_j-Y_i=-\frac{\beta}{\alpha}(y_j-y_i).
\]
Therefore replacing \(x,y\) by \(X,Y\) only multiplies all entries in the rows \(1i\), \(i\in I\), by the common non-zero scalar \(-1/\beta\), and all entries in the rows \(2i\), \(i\in I\), by the common non-zero scalar \(-\alpha/\beta\).  These row scalings, together with the inverse scalings of the corresponding basis columns, do not change the represented matroid and keep the basis block equal to the identity.  Hence \(S(X,Y)\) and \(S(x,y)\) represent the same matroid.
\end{proof}
\subsection{Conversion to the hyperconnectivity matrix}
We now convert the normalized star-basis representation into the two-dimensional hyperconnectivity matrix.
\begin{proposition}\label{prop:hyper-conversion}
Let \(x_i\in\F^\times\), \(i\in I\), be distinct, and set
\[
y_i=\frac{1}{x_i}
\qquad (i\in I).
\]
Then the column matroid of \(S(x,y)\) is represented by the hyperconnectivity matrix
\[
H(q_1,\ldots,q_n),
\]
where
\begin{equation*}
q_1=(c_1,d_1)=(1,0),
\qquad
q_2=(c_2,d_2)=(0,1),
\end{equation*}
and
\begin{equation*}
q_i=(c_i,d_i)=(1,x_i)=\left(1,\frac{1}{y_i}\right)
\qquad (i\in I).
\end{equation*}
\end{proposition}
\begin{proof}
Let \(h_{ij}\) denote the column of \(H(q_1,\ldots,q_n)\) indexed by the edge \(ij\).  The double-star set \(B\) is independent in \(H(q_1,\ldots,q_n)\).  Indeed, suppose
\[
\lambda h_{12}+\sum_{i\in I}\alpha_i h_{1i}+\sum_{i\in I}\beta_i h_{2i}=0.
\]
For a fixed \(i\in I\), the rows \(C_i,D_i\) give
\[
-\alpha_i=0,
\qquad
-\beta_i=0.
\]
Thus \(\alpha_i=\beta_i=0\) for all \(i\in I\), and then \(\lambda h_{12}=0\), so \(\lambda=0\).

Fix \(i,j\in I\) with \(i<j\).  We claim that
\begin{equation}\label{eq:hyper-column-relation}
h_{ij}=(x_i-x_j)h_{12}-h_{1i}+h_{1j}-x_jh_{2i}+x_ih_{2j}.
\end{equation}
This is checked row by row.  In rows \(C_i,D_i\), the right hand side of \eqref{eq:hyper-column-relation} is
\[
-h_{1i}-x_jh_{2i}
=-(-1,0)-x_j(0,-1)
=(1,x_j),
\]
which is the entry of \(h_{ij}\) in rows \(C_i,D_i\).  In rows \(C_j,D_j\), it is
\[
h_{1j}+x_ih_{2j}
=(-1,0)+x_i(0,-1)
=(-1,-x_i),
\]
which is the entry of \(h_{ij}\) in rows \(C_j,D_j\).  In rows \(C_1,D_1\), it is
\[
(x_i-x_j)(0,1)-(1,x_i)+(1,x_j)=(0,0),
\]
and in rows \(C_2,D_2\), it is
\[
(x_i-x_j)(-1,0)-x_j(1,x_i)+x_i(1,x_j)=(0,0).
\]
All other rows are zero.  This proves \eqref{eq:hyper-column-relation}.

Now perform the non-zero column scalings
\begin{equation*}
\widehat h_{12}=h_{12},
\qquad
\widehat h_{1i}=-h_{1i},
\qquad
\widehat h_{2i}=-\frac{1}{x_i}h_{2i}
\qquad (i\in I).
\end{equation*}
Then \eqref{eq:hyper-column-relation} becomes
\begin{equation*}
h_{ij}=(x_i-x_j)\widehat h_{12}
+\widehat h_{1i}-\widehat h_{1j}
+x_ix_j(\widehat h_{2i}-\widehat h_{2j}).
\end{equation*}
Finally scale the non-basis column \(h_{ij}\) by \((x_i-x_j)^{-1}\).  In the basis \(B\), the resulting column has coordinates
\[
\widehat h_{12}
+\frac{\widehat h_{1i}-\widehat h_{1j}}{x_i-x_j}
+\frac{x_ix_j}{x_i-x_j}(\widehat h_{2i}-\widehat h_{2j}).
\]
Since
\[
\frac{x_ix_j}{x_i-x_j}
=\frac{1}{\frac{1}{x_j}-\frac{1}{x_i}}
=\frac{1}{y_j-y_i},
\]
this is exactly the column \eqref{eq:star-column} of \(S(x,y)\).  Therefore \(H(q_1,\ldots,q_n)\) and \(S(x,y)\) represent the same matroid.
\end{proof}

\begin{proposition}[Normalized hyperconnectivity realization]
	\label{prop:normalized-hyper-realization}
	Let $\F$ be a field, let $n\geq6$, and put $I=\{3,\ldots,n\}$.
	Suppose that $M$ is an $\F$-representable abstract $2$-rigidity matroid
	on $\binom{[n]}{2}$ in which every copy of $K_{3,3}$ is a circuit.
	Then there are distinct elements
	\[
	x_i\in\F^\times\qquad(i\in I)
	\]
	such that $M$ is represented by
	\[
	H(q_1,\ldots,q_n),
	\]
	where
	\[
	q_1=(1,0),\qquad q_2=(0,1),\qquad
	q_i=(1,x_i)\quad(i\in I).
	\]
\end{proposition}
\begin{proof}
	By Theorem~\ref{thm:abstract-normal-form}, the matroid $M$ is
	represented by a star-basis matrix $S(x,y)$. By
	Proposition~\ref{prop:k33-fractional}, after replacing $S(x,y)$ by an
	equivalent star-basis matrix, we may assume that
	\[
	y_i=\frac{1}{x_i}\qquad(i\in I),
	\]
	where the elements $x_i\in\F^\times$, $i\in I$, are distinct.
	Proposition~\ref{prop:hyper-conversion} then gives the stated
	hyperconnectivity representation.
\end{proof}
\begin{proof}[Proof of Theorem~\ref{thm:main-hyper-realization}]
	The theorem follows immediately from
	Proposition~\ref{prop:normalized-hyper-realization}.
\end{proof}
\begin{remark}
Over $\mathbb{R}$, Theorem~\ref{thm:main-hyper-realization} can also be deduced directly from Theorem~\ref{thm:rigidity-realization} roughly as follows (I thank D\'aniel Garamv\"olgyi for bringing this to my attention). By the Bolker--Roth theorem~\cite{BR80}, a copy of $K_{3,3}$ is dependent in the infinitesimal rigidity matrix if and only if its six points lie on a conic. Hence, the matroid $M$, by Theorem~\ref{thm:rigidity-realization} and the Bolker--Roth theorem, will have an infinitesimal rigidity matrix realization over $\mathbb{R}$, such that all underlying points lie on the same non-degenerate conic. Using projective invariance~\cite{NSW}, we may then assume this conic to be the standard parabola $y=x^2$. With these coordinates, one can apply row and column transformations to the infinitesimal rigidity matrix, to bring it into hyperconnectivity form.  
\end{remark}
\subsection{A restriction on further circuits}
In this subsection we prove Theorem~\ref{thm:representable-families}.
For this, let us first recall two elementary facts about polynomials.
\begin{lemma}\label{lem:grid}
Let \(X\subseteq\F\) be finite, and let
\[
P\in\F[z_1,\ldots,z_v]
\]
be a polynomial satisfying
\[
\deg_{z_r}P<|X|
\qquad (1\leq r\leq v).
\]
If \(P\) vanishes on \(X^v\), then \(P=0\).
\end{lemma}
\noindent
The proof is a straightforward induction on \(v\), we omit the details.
\begin{lemma}\label{lem:injective-grid}
Let \(P\in\F[z_1,\ldots,z_v]\) be non-zero, and suppose that
\[
\deg_{z_r}P\leq d
\qquad (1\leq r\leq v).
\]
Let \(X\subseteq\F\) be finite.  If
\[
P(t_1,\ldots,t_v)=0
\]
for every ordered \(v\)-tuple \((t_1,\ldots,t_v)\) of distinct elements of \(X\), then
\begin{equation}\label{eq:X-bound}
|X|\leq d+v-1.
\end{equation}
\end{lemma}
\begin{proof}
Let
\[
\Delta(z_1,\ldots,z_v)=\prod_{1\leq r<s\leq v}(z_r-z_s)
\]
be the Vandermonde polynomial, and put
\[
Q=P\Delta.
\]
Since \(\F[z_1,\ldots,z_v]\) is an integral domain, \(Q\neq0\).  If \((t_1,\ldots,t_v)\in X^v\), then either the \(t_r\) are distinct, in which case \(P(t_1,\ldots,t_v)=0\), or two entries are equal, in which case \(\Delta(t_1,\ldots,t_v)=0\).  Thus \(Q\) vanishes on \(X^v\).

For each \(r\), the degree of \(\Delta\) in \(z_r\) is \(v-1\).  Hence
\[
\deg_{z_r} Q\leq d+v-1
\qquad (1\leq r\leq v).
\]
If \(|X|>d+v-1\), then \(\deg_{z_r}Q<|X|\) for every \(r\), so Lemma~\ref{lem:grid} implies \(Q=0\), a contradiction.  Therefore \eqref{eq:X-bound} holds.
\end{proof}

\begin{definition}\label{def:symmetric}
	A matroid \(M\) on \(\binom{[n]}{2}\) is \emph{symmetric} if its independent sets, and equivalently its circuits, are invariant under the permutations of $\binom{[n]}{2}$ induced by permutations of $[n]$.  If \(F\) is a graph with \(|V(F)|\leq n\), we say that \(F\) is \(M\)-independent/dependent/circuit if every copy of \(F\) on \([n]\) is of that status in \(M\).
\end{definition}
\noindent
We next compare $H_{2,n}$ with hyperconnectivity matrices obtained by
specializing their parameters.
\begin{lemma}\label{lem:generic-specialization}
	Suppose that $\F$ is a field and that $q_1,\ldots,q_n\in\F^2$. Then $H_{2,n}$ is freer than the column matroid of $H(q_1,\ldots,q_n)$.
\end{lemma}
\begin{proof}
	Write $q_i=(\gamma_i,\delta_i)$ for every $i\in[n]$. Let
	$J\subseteq\binom{[n]}{2}$ be independent in the column matroid of
	$H(q_1,\ldots,q_n)$. There is a choice of $|J|$ rows such that the
	resulting square submatrix with columns indexed by $J$ has non-zero
	determinant. Let
	\[
	D\in\F[c_1,d_1,\ldots,c_n,d_n]
	\]
	be the determinant of the corresponding submatrix of
	\[
	H\bigl((c_1,d_1),\ldots,(c_n,d_n)\bigr).
	\]
	By the choice of the rows, we have
	\[
	D(\gamma_1,\delta_1,\ldots,\gamma_n,\delta_n)\neq0.
	\]
	Hence $D$ is a non-zero polynomial, and consequently the columns
	indexed by $J$ are independent over
	$\F(c_1,d_1,\ldots,c_n,d_n)$. By Lemma~\ref{lem:scalar-extension} and the definition of $H_{2,n}$,
	they are independent in $H_{2,n}$.
	Thus $H_{2,n}$ is freer than the column matroid of
	$H(q_1,\ldots,q_n)$.
\end{proof}
Let \(F\) be a graph with vertex set \([v]=\{1,\ldots,v\}\) and edge set \(E(F)\).  Put \(m=|E(F)|\).  Define a \(2v\times m\) matrix
\[
A_F(z)=A_F(z_1,\ldots,z_v)
\]
over \(\F[z_1,\ldots,z_v]\) as follows.  The rows are indexed by
\[
C_1,D_1,\ldots,C_v,D_v,
\]
and the columns are indexed by \(E(F)\).  For an edge \(ab\in E(F)\), with \(a<b\), the column indexed by \(ab\) has non-zero entries
\begin{equation*}
	(A_F)_{C_a,ab}=1,
	\qquad
	(A_F)_{D_a,ab}=z_b,
	\qquad
	(A_F)_{C_b,ab}=-1,
	\qquad
	(A_F)_{D_b,ab}=-z_a.
\end{equation*}
Thus \(A_F(z)\) is the two-dimensional hyperconnectivity matrix
$
H((1,z_1),\ldots,(1,z_v))
$,
restricted in the columns to the edges of \(F\).
\begin{lemma}\label{lem:PF-polynomial}
	Suppose that $\F$ is a field and that $G$ is an $\mathcal H_2$-independent graph with vertex set $[v]$ and $m$ edges. Then there exists a non-zero polynomial
	\[
	P_{G,\F}(z_1,\ldots,z_v)\in\F[z_1,\ldots,z_v]
	\]
	such that
	\begin{enumerate}
		\item[(i)] 	$
		\deg_{z_r}P_{G,\F}\leq\deg_G(r)
		\qquad(1\leq r\leq v), $
		\item[(ii)]  For all $(\xi_1,\ldots,\xi_v)\in\F^v$ such that the columns of
		$A_G(\xi_1,\ldots,\xi_v)$ are dependent, we have $P_{G,\F}(\xi_1,\ldots,\xi_v)=0$.
	\end{enumerate}	
\end{lemma}
\begin{proof}
	Let $c_1,d_1,\ldots,c_v,d_v$ be algebraically independent indeterminates
	over $\F$. By Lemma~\ref{lem:scalar-extension} and the definition of $H_{2,v}$,
	the matrix
	\[
	H\bigl((c_1,d_1),\ldots,(c_v,d_v)\bigr)
	\]
	represents $H_{2,v}$ over
	$\F(c_1,d_1,\ldots,c_v,d_v)$. Since $G$ is
	$\mathcal H_2$-independent, its columns indexed by $E(G)$ are independent
	over this field. Put
	\[
	z_r=\frac{d_r}{c_r}
	\qquad(1\leq r\leq v).
	\]
	The elements $z_1,\ldots,z_v$ are algebraically independent over $\F$, and
	\[
	\F(c_1,d_1,\ldots,c_v,d_v)
	=
	\F(c_1,z_1,\ldots,c_v,z_v).
	\]
	Multiply the rows $C_r,D_r$ by $c_r$ for every $r\in[v]$, and multiply the column indexed by $ab$ by $(c_ac_b)^{-1}$ for every $ab\in E(G)$. If $a<b$, the resulting column has non-zero entries
	\[
	C_a:1,\qquad
	D_a:z_b,\qquad
	C_b:-1,\qquad
	D_b:-z_a.
	\]
	Thus the resulting matrix is $A_G(z_1,\ldots,z_v)$. Its columns are independent over $\F(c_1,z_1,\ldots,c_v,z_v)$. Since its entries belong to the subfield $\F(z_1,\ldots,z_v)$, they are also independent over that subfield.
	
	Consequently, some $m\times m$ minor of $A_G(z_1,\ldots,z_v)$ is non-zero. Fix one such minor and denote its determinant by $P_{G,\F}$. The variable $z_r$ occurs only in columns indexed by edges incident with $r$, and it occurs at most linearly in each such column. Since a determinant is multilinear in its columns, we have
	\[
	\deg_{z_r}P_{G,\F}\leq\deg_G(r)
	\qquad(1\leq r\leq v).
	\]
	Finally, suppose that $\xi_1,\ldots,\xi_v\in\F$ and that the columns of $A_G(\xi_1,\ldots,\xi_v)$ are dependent. Then all their $m\times m$ minors vanish, including the chosen minor $P_{G,\F}$. Therefore
	\[
	P_{G,\F}(\xi_1,\ldots,\xi_v)=0.
	\]
\end{proof}
\begin{theorem}\label{thm:finite-bound}
	Suppose that $G$ is an $\mathcal H_2$-independent graph with $|V(G)|=v$, and let $\Delta(G)$ be its maximum degree. Let $M$ be a linearly representable symmetric abstract $2$-rigidity matroid on $n\geq6$ vertices such that every copy of $K_{3,3}$ is an $M$-circuit. If $G$ is $M$-dependent, then
	\[
	n\leq\Delta(G)+v+1.
	\]
\end{theorem}

\begin{proof}
Let $\F$ be a field over which $M$ is represented, and put
$I=\{3,\ldots,n\}$. By
Proposition~\ref{prop:normalized-hyper-realization}, there are distinct
elements $x_i\in\F^\times$, $i\in I$, such that $M$ is represented by
\[
H(q_1,\ldots,q_n),
\]
where
\[
q_1=(1,0),\qquad q_2=(0,1),\qquad
q_i=(1,x_i)\quad(i\in I).
\]
Put $X=\{x_i:i\in I\}$; we have $|X|=n-2$. Apply Lemma~\ref{lem:PF-polynomial} over $\F$ after labelling the vertices of $G$ by $[v]$. We obtain a non-zero polynomial
\[
P_{G,\F}(z_1,\ldots,z_v)\in\F[z_1,\ldots,z_v]
\]
satisfying
\[
\deg_{z_r}P_{G,\F}\leq\deg_G(r)\leq\Delta(G)
\qquad(1\leq r\leq v).
\]
Let $(t_1,\ldots,t_v)$ be an ordered $v$-tuple of distinct elements of $X$. Since the $x_i$ are distinct, there is an injective map $\phi:[v]\longrightarrow I$ such that
	\[
	t_r=x_{\phi(r)}
	\qquad(1\leq r\leq v).
	\]
The copy $\phi(G)$ is dependent in $M$ because $G$ is
$M$-dependent. After relabelling the rows $C_r,D_r$ of
$A_G(t_1,\ldots,t_v)$ as $C_{\phi(r)},D_{\phi(r)}$, this matrix is
obtained from the restriction of $H(q_1,\ldots,q_n)$ to the columns
indexed by $\phi(G)$ by deleting rows that are zero on all these
columns and multiplying some columns by $-1$. These operations preserve
all column dependences. Hence the columns of
$A_G(t_1,\ldots,t_v)$ are dependent, and
Lemma~\ref{lem:PF-polynomial} gives
\[
P_{G,\F}(t_1,\ldots,t_v)=0.
\]
	Thus $P_{G,\F}$ vanishes on every ordered $v$-tuple of distinct elements of $X$. Applying Lemma~\ref{lem:injective-grid} with $d=\Delta(G)$ gives
	\[
	|X|\leq\Delta(G)+v-1.
	\]
	Since $|X|=n-2$, we obtain
	\[
	n\leq\Delta(G)+v+1.
	\]
\end{proof}

\begin{proof}[Proof of Theorem~\ref{thm:representable-families}]
	Let $M_n$ be the matroid on $\binom{[n]}{2}$ whose independence family is $\mathcal F^n$. By the definition of a $2$-rigidity family, each $M_n$ is symmetric. Moreover, the status of a fixed graph as independent/dependent/circuit in $M_n$ does not depend on $n$ once $n$ is at least its number of vertices.
	
	Suppose, for a contradiction, that $M_n$ is linearly representable for arbitrarily large $n$. Since $\mathcal F\neq\mathcal R_2$, it follows from~\cite[Theorem~1.6]{Tyo26} that $K_{3,3}$ is a circuit in $\mathcal F$. Hence every copy of $K_{3,3}$ is an $M_n$-circuit whenever $n\geq6$.
	
	We claim that
	\[
	\mathcal F\subseteq\mathcal H_2.
	\]
	Let $G\in\mathcal F$, and put $v=|V(G)|$. Choose $n\geq\max\{6,v\}$ such that $M_n$ is linearly representable, and let $\F$ be a field over which it is represented. By Theorem~\ref{thm:main-hyper-realization}, the matroid $M_n$ has a representation by a hyperconnectivity matrix over $\F$. Lemma~\ref{lem:generic-specialization} gives
	\[
	M_n\preceq H_{2,n}.
	\]
	Every copy of $G$ is independent in $M_n$, so it is independent in $H_{2,n}$, i.e. $G$ is $\mathcal H_2$-independent. This proves the claimed inclusion.
	
	Since $\mathcal F\neq\mathcal H_2$, choose a graph $G\in\mathcal H_2\setminus\mathcal F$, and put $v=|V(G)|$. Every copy of $G$ is dependent in $M_n$ whenever $n\geq v$. Choose a linearly representable $M_n$ satisfying
	\[
	n>\max\{6,\Delta(G)+v+1\}.
	\]
	Then $M_n$ satisfies all the hypotheses of Theorem~\ref{thm:finite-bound}, which gives
	\[
	n\leq\Delta(G)+v+1,
	\]
	contradicting the choice of $n$. 
	
	Therefore, there exists $n_0=n_0(\mathcal F)$ such that $M_n$ is not linearly representable whenever $n\geq n_0$, as claimed.
\end{proof}
\section*{Acknowledgement}
Similar results were obtained, independently and using different methods, by Garamv\"olgyi and Vermant, and will appear in a forthcoming paper~\cite{GV}. I thank D\'aniel Garamv\"olgyi for insightful comments on the manuscript.
\bibliographystyle{amsplain}
\bibliography{refs}

\end{document}